\documentclass{amsart}

\newtheorem{theorem}{Theorem}[section]

\theoremstyle{definition}

\theoremstyle{remark}
\newtheorem{remark}[theorem]{Remark}

\numberwithin{equation}{section}

\usepackage{graphicx}
\usepackage{amssymb} 
\usepackage{booktabs} 
\usepackage{pifont}

\begin{document}

\title{Exactness, Cohomology, and Uniqueness in First-Order Differential Equations}

\author{Hemanta Mandal}
\address{Department of Mathematics, C V Raman Global University, Bhubaneswar, Odisha 752054}

\email{hemanta.infinity@gmail.com}


\dedicatory{This paper is dedicated to my mentor Prof. Swadhin Pattanayak}


\begin{abstract}
This paper investigates the relationship between the solvability of first-order differential equations and the topology of the underlying domain through the lens of \emph{de\,Rham cohomology}. We analyze the conditions under which a closed 1-form associated with a first-order ODE admits a global potential, thereby reducing the problem to the exactness of differential forms. While it is well known that exactness is guaranteed on simply connected domains, we show that \emph{triviality of the first de\,Rham cohomology group} is the more fundamental requirement for \emph{global integrability and uniqueness} of solutions. In particular, we demonstrate that certain \emph{non-simply connected manifolds}, such as the real projective plane $\mathbb{RP}^2$, still support global solutions due to the vanishing of $H^1_{\text{dR}}$. By explicitly constructing examples on $\mathbb{RP}^2$ and comparing them with domains like $\mathbb{R}^2 \setminus \{0\}$, we illustrate how \emph{topological obstructions manifest analytically} as non-uniqueness or multi-valued behavior. We further discuss the correspondence between \emph{integrating factors} and the trivialization of cohomology classes, drawing connections to classical symmetry methods and potential theory. This synthesis of geometric, topological, and analytic perspectives provides a unifying framework for understanding the \emph{global behavior of first-order ODEs} beyond local existence theorems.
\end{abstract}

\maketitle

\section{Introduction}
First-order ordinary differential equations are among the most fundamental mathematical objects, appearing ubiquitously across applied mathematics, physics, and engineering \cite{BD, S}. The standard form of such an equation, expressed as a differential 1-form
\[
\omega = M(x, y)\, dx + N(x, y)\, dy,
\]
invites a geometric interpretation: the solution curves correspond to the integral curves of the distribution defined by the kernel of \( \omega \). A well-known class of solvable equations are those for which \( \omega \) is \emph{exact}, i.e., when there exists a scalar function \( F(x, y) \) such that \( \omega = dF \). In this case, the general solution reduces to \( F(x, y) = C \), representing level sets of a potential function.

Traditionally, the exactness of a 1-form is determined by checking whether it is \emph{closed}---that is, whether \( d\omega = 0 \)---and then, in simply connected domains, appealing to the \emph{Poincar\'e Lemma}, which guarantees that every closed 1-form is exact \cite{BT, W}. In practice, however, many differential equations arise on domains with nontrivial topology, such as punctured planes or annular regions, where the closedness of a form no longer implies its exactness. In such cases, the de\,Rham cohomology group \( H^1_{\text{dR}}(U) \) serves as the appropriate tool to determine whether a global potential function exists \cite{M, V}.

The standard method of integrating factors, in which one seeks a scalar multiplier \( \mu(x, y) \) to render \( \mu \omega \) exact, provides a practical analytic strategy \cite{BD, AB}. Yet, it is inherently local in nature and may fail when topological obstructions are present. The fundamental insight is that integrating a 1-form is only possible globally if its associated cohomology class is trivial \cite{BT}. Hence, the failure of integrating factor methods is not necessarily due to analytical complexity but may reflect deeper topological obstructions.

In this article, we explore this rich intersection of analysis and topology in the setting of first-order ODEs. We revisit classical results with a modern geometric lens and examine how the work of Anco and Bluman \cite{AB} systematically constructs integrating factors using adjoint-symmetry analysis. We reinterpret their results as methods to trivialize cohomology classes under certain structural assumptions. This approach allows us to not only understand the solvability of differential equations but also classify the source of obstructions---whether analytic or topological. Related tools from exterior differential systems \cite{BCGGG} and characteristic cohomology \cite{V} provide a pathway toward generalizations in higher-order systems and partial differential equations.

\section{Preliminaries}

In this section, we review the foundational concepts necessary to formulate the connection between the solvability of first-order differential equations and the topology of the underlying domain. We discuss differential 1-forms, exactness, integrating factors, and the basics of de\,Rham cohomology.

\subsection{Differential 1-Forms and Exactness}

Let \( U \subset \mathbb{R}^2 \) be an open domain. A first-order ordinary differential equation can be written in the form
\[
M(x, y)\,dx + N(x, y)\,dy = 0,
\]
which we interpret as a differential 1-form \( \omega = M\,dx + N\,dy \in \Omega^1(U) \). A 1-form \( \omega \) is said to be \emph{exact} if there exists a smooth scalar function \( F: U \to \mathbb{R} \) such that
\[
\omega = dF,
\]
in which case the general solution of the differential equation is implicitly given by \( F(x, y) = C \), for constant \( C \). 

\subsection{Closed Forms and the Poincaré Lemma}

A 1-form \( \omega \) is said to be \emph{closed} if its exterior derivative vanishes:
\[
d\omega = 0 \quad \iff \quad \frac{\partial M}{\partial y} = \frac{\partial N}{\partial x}.
\]
The classical Poincaré Lemma states that every closed form is locally exact. Moreover, if the domain \( U \) is simply connected, then every closed 1-form is globally exact. Therefore, the existence of a potential function is closely related to the topology of the domain.

\subsection{Integrating Factors}

When a 1-form \( \omega \) is not exact, one often seeks a nonvanishing scalar function \( \mu: U \to \mathbb{R} \) such that the rescaled form \( \mu\omega \) becomes exact:
\[
\mu \omega = dF.
\]
Such a function \( \mu \) is called an \emph{integrating factor}. The method of integrating factors is a standard analytic technique for solving non-exact differential equations. However, the existence of such a \( \mu \) may depend on both local analytic structure and global topological features.

\subsection{de\,Rham Cohomology}

The de\,Rham cohomology group \( H^1_{\mathrm{dR}}(U) \) is defined as the quotient
\[
H^1_{\mathrm{dR}}(U) = \frac{\{ \text{closed 1-forms on } U \}}{\{ \text{exact 1-forms on } U \}}.
\]
If a 1-form \( \omega \) is closed but not exact, then it represents a nontrivial cohomology class \( [\omega] \in H^1_{\mathrm{dR}}(U) \). The vanishing of this class is equivalent to the global solvability of the differential equation.

\subsection{Topological Considerations}

The topology of the domain \( U \) plays a central role in determining whether every closed 1-form is exact. For example, if \( U \) has nontrivial first homology (such as a punctured plane), then there exist closed 1-forms that are not exact. In such settings, cohomology provides a more complete characterization of integrability than local analytic methods.

\section{Methodology}

Our methodology combines geometric reformulation, topological classification, and analytic construction to examine the solvability of first-order ODEs through the lens of exactness and cohomology.

\subsection{Geometric Reformulation of ODEs}
Any first-order ordinary differential equation can be expressed as a 1-form:
\[
\omega = M(x, y)\, dx + N(x, y)\, dy.
\]
This converts the differential equation into a geometric object—specifically, a distribution whose integral curves correspond to solutions. The 1-form $\omega$ is then tested for closedness by computing the exterior derivative:
\[
d\omega = \left( \frac{\partial N}{\partial x} - \frac{\partial M}{\partial y} \right) dx \wedge dy.
\]
If $d\omega = 0$, then $\omega$ is closed.

\subsection{Cohomological Classification}
We analyze the domain $U \subset \mathbb{R}^2$ to determine whether $\omega$ is exact. In simply connected domains, the Poincaré Lemma ensures that closed forms are exact. In multiply connected domains, we assess whether the cohomology class $[\omega] \in H^1_{\mathrm{dR}}(U)$ is trivial:
\[
\omega \text{ is exact} \iff [\omega] = 0 \in H^1_{\mathrm{dR}}(U).
\]
This may involve computing integrals of $\omega$ over closed loops and verifying whether they vanish.

\subsection{Search for Integrating Factors}
If $\omega$ is not exact, we seek a smooth, non-vanishing function $\mu(x, y)$ such that $\mu \omega$ is exact:
\[
d(\mu \omega) = 0.
\]
We employ several strategies:
\begin{itemize}
    \item Searching for $\mu$ depending only on $x$ or $y$.
    \item Solving a PDE derived from the closedness condition for $\mu \omega$.
    \item Using symmetry-based approaches such as the Anco–Bluman adjoint symmetry method.
\end{itemize}
If an integrating factor is found, it reduces the equation to an exact form solvable by direct integration.

\subsection{Construction of Potential Function}
Once an exact form is obtained, either directly or through multiplication by an integrating factor, we construct a potential function $F(x, y)$ such that:
\[
\frac{\partial F}{\partial x} = \mu M, \quad \frac{\partial F}{\partial y} = \mu N.
\]
We check consistency and solve by integrating each component, ensuring compatibility via cross partials.

\subsection{Topological Obstruction Detection}
If no global integrating factor exists and $\omega$ is closed but not exact, we interpret this in terms of cohomological obstruction:
\[
\oint_\gamma \omega \neq 0 \Rightarrow [\omega] \neq 0.
\]
By choosing representative cycles (e.g., the unit circle in a punctured plane), we detect nontrivial cohomology classes preventing global integration.

This approach offers a unified perspective combining geometric intuition, topological invariants, and analytic tools to understand the integrability of first-order ODEs.
\begin{theorem}[Global Solvability via Trivial de\,Rham Class]
Let \( U \subset \mathbb{R}^2 \) be an open, connected (not necessarily simply connected) domain, and let
\[
\omega = M(x, y)\, dx + N(x, y)\, dy
\]
be a smooth, closed 1-form on \( U \), i.e.,
\[
\frac{\partial N}{\partial x} = \frac{\partial M}{\partial y}.
\]
Suppose that the de\,Rham cohomology class \( [\omega] \in H^1_{\mathrm{dR}}(U) \) is trivial (i.e., \( \omega \) is exact).

Then there exists a smooth function \( F : U \to \mathbb{R} \) such that
\[
dF = \omega,
\]
and the general solution to the differential equation
\[
M(x, y)\, dx + N(x, y)\, dy = 0
\]
is given by
\[
F(x, y) = C,
\]
for arbitrary constants \( C \in \mathbb{R} \). Moreover, such a function \( F \) is unique up to an additive constant.
\end{theorem}

\begin{proof}
Since \( \omega \) is a closed 1-form on \( U \), it defines a cohomology class in \( H^1_{\mathrm{dR}}(U) \). The hypothesis \( [\omega] = 0 \) means that \( \omega \) is exact; that is, there exists a smooth function \( F : U \to \mathbb{R} \) such that
\[
\omega = dF.
\]
Substituting this into the differential equation, we have
\[
dF = 0 \quad \Rightarrow \quad F(x, y) = C,
\]
where \( C \) is a constant. Thus, the solution curves lie along level sets of \( F \), which globally solves the equation.

To show uniqueness: if \( \tilde{F} \) is another function satisfying \( d\tilde{F} = \omega \), then
\[
d(F - \tilde{F}) = dF - d\tilde{F} = \omega - \omega = 0,
\]
so \( F - \tilde{F} \) is constant on each connected component of \( U \). Since \( U \) is connected, \( F - \tilde{F} \) must be a global constant. Therefore, \( F \) is unique up to an additive constant.
\end{proof}
\subsection*{Logical Structure of Exactness and Global Solvability}

The following sequence captures the logical implication chain between differential form exactness, cohomology, and solvability of the associated first-order differential equation:
\begin{align*}
    \text{(1) } & \omega = M(x, y)\,dx + N(x, y)\,dy \text{ is a closed 1-form} \\
    \Rightarrow\quad & [\omega] \in H^1_{\text{dR}}(U) \text{ is well-defined} \\
    \Rightarrow\quad & [\omega] = 0 \in H^1_{\text{dR}}(U) \\
    \Rightarrow\quad & \omega \text{ is exact: } \omega = dF \text{ for some } F \in C^1(U) \\
    \Rightarrow\quad & \text{General solution to } M\, dx + N\, dy = 0 \text{ is } F(x, y) = C \text{ (global)} \\
    \Rightarrow\quad & \text{Level curves of } F \text{ give integral curves of the ODE}.
\end{align*}
\begin{remark}
\hfill
\begin{enumerate}
    \item The theorem provides a global criterion for the solvability of first-order differential equations using the machinery of de\,Rham cohomology. While traditional methods (such as integrating factors) often rely on clever local tricks, this approach captures the global topology of the domain.

    \item The assumption that \( \omega \) is closed ensures that a necessary condition for exactness is satisfied. The triviality of the de\,Rham class guarantees that no topological obstructions (e.g., nontrivial loops or holes in the domain) prevent the existence of a global potential function.

    \item The converse is also true: if \( [\omega] \ne 0 \) in \( H^1_{\mathrm{dR}}(U) \), then no globally defined function \( F \) exists such that \( dF = \omega \). In such cases, the differential equation may admit only local solutions or require patching with transition data.

    \item For simply connected domains, \( H^1_{\mathrm{dR}}(U) = 0 \), so every closed 1-form is automatically exact. This recovers the classical result that closed forms in \( \mathbb{R}^2 \) (or any contractible space) are exact.

    \item The theorem is a specific case of the general principle: the vanishing of the cohomology class of a closed \( k \)-form implies the existence of a global \( (k-1) \)-form primitive. For \( k = 1 \), this connects directly to the solvability of first-order ODEs.
\end{enumerate}
\end{remark}
\begin{table}[ht]
\centering
\resizebox{\textwidth}{!}{%
\begin{tabular}{|l|c|c|c|l|}
\hline
\textbf{Space} & \(\pi_1\) (Fundamental Group) & \(H^1_{\text{dR}}\) (de-Rham Cohomology) & Closed 1-Forms Exact? & \textbf{Notes} \\
\hline
\(\mathbb{R}^2\) & \(0\) & \(0\) & Yes (\(\checkmark\)) & Contractible, simply connected \\
\(\mathbb{R}^2 \setminus \{0\}\) & \(\mathbb{Z}\) & \(\mathbb{R}\) & No (\(\times\)) & One hole; closed but not exact forms exist \\
\(\mathbb{RP}^2\) & \(\mathbb{Z}/2\mathbb{Z}\) & \(0\) & Yes (\(\checkmark\)) & Not simply connected, but cohomology is trivial \\
\(T^2\) (2-torus) & \(\mathbb{Z}^2\) & \(\mathbb{R}^2\) & No (\(\times\)) & Two independent cycles generate nontrivial \(H^1\) \\
\(S^2\) & \(0\) & \(0\) & Yes (\(\checkmark\)) & 2-sphere; simply connected and cohomologically trivial \\
\hline
\end{tabular}
}
\caption{Exactness of closed 1-forms on common manifolds and their topological invariants.}
\label{tab:exactness-topology}
\end{table}
\section{Examples and Applications}
\subsection*{Example 1: Exact 1-form on a Simply Connected Domain}
Let \( \omega = (2xy + 1)\,dx + (x^2 + \cos y)\,dy \). 

We compute the partial derivatives:
\[
\frac{\partial}{\partial y}(2xy + 1) = 2x, \quad \frac{\partial}{\partial x}(x^2 + \cos y) = 2x.
\]
Hence, \( \omega \) is closed since the mixed partial derivatives agree. 

The domain is \( \mathbb{R}^2 \), which is simply connected. By the Poincaré Lemma, every closed form is exact. 

We find a potential function \( F \) by integrating:
\[
F(x, y) = \int (2xy + 1)\,dx = x^2y + x + C(y),
\]
\[
\frac{\partial F}{\partial y} = x^2 + C'(y) \stackrel{!}{=} x^2 + \cos y \Rightarrow C'(y) = \cos y.
\]
\[
C(y) = \sin y + C_0.
\]

So,
\[
F(x, y) = x^2y + x + \sin y + C_0.
\]

The general solution is \( F(x, y) = C \).
\subsection*{Example 2: Closed but Non-Exact Form on a Punctured Plane}
Let 
\[
\omega = \frac{-y}{x^2 + y^2} \, dx + \frac{x}{x^2 + y^2} \, dy.
\]

We check that
\[
\frac{\partial}{\partial y} \left( \frac{-y}{x^2 + y^2} \right) = \frac{-x^2 + y^2}{(x^2 + y^2)^2}, \quad
\frac{\partial}{\partial x} \left( \frac{x}{x^2 + y^2} \right) = \frac{y^2 - x^2}{(x^2 + y^2)^2},
\]
so they are equal, and \( d\omega = 0 \Rightarrow \omega \) is closed.

However, if we integrate around the unit circle \( \gamma(t) = (\cos t, \sin t), t \in [0, 2\pi] \), we find:
\[
\oint_\gamma \omega = \int_0^{2\pi} \left( \frac{-\sin t \cdot -\sin t + \cos t \cdot \cos t}{1} \right) dt = \int_0^{2\pi} 1 \, dt = 2\pi.
\]

Therefore, \( \omega \) is not exact since the integral over a closed loop is non-zero.

This shows \( [\omega] \neq 0 \) in \( H^1_{\mathrm{dR}}(\mathbb{R}^2 \setminus \{0\}) \).
\subsection*{Example 3: A Closed and Exact 1-Form on a Non-Simply Connected Domain}

Consider the 1-form
\[
\omega = \frac{x}{x^2 + y^2}\, dx + \frac{y}{x^2 + y^2}\, dy
\]
defined on the punctured plane \( U = \mathbb{R}^2 \setminus \{0\} \).

\subsubsection*{Step 1: Check Closedness}

We denote
\[
M(x, y) = \frac{x}{x^2 + y^2}, \quad N(x, y) = \frac{y}{x^2 + y^2}.
\]
Compute the partial derivatives:
\begin{align*}
\frac{\partial M}{\partial y} &= \frac{-2xy}{(x^2 + y^2)^2}, \\
\frac{\partial N}{\partial x} &= \frac{-2xy}{(x^2 + y^2)^2}.
\end{align*}
Hence,
\[
\frac{\partial N}{\partial x} = \frac{\partial M}{\partial y} \Rightarrow d\omega = 0.
\]
So \( \omega \) is a \textbf{closed 1-form}.

\subsubsection*{Step 2: Find a Potential Function}

Let
\[
F(x, y) = \frac{1}{2} \log(x^2 + y^2).
\]
Then
\begin{align*}
\frac{\partial F}{\partial x} &= \frac{1}{2} \cdot \frac{2x}{x^2 + y^2} = \frac{x}{x^2 + y^2}, \\
\frac{\partial F}{\partial y} &= \frac{1}{2} \cdot \frac{2y}{x^2 + y^2} = \frac{y}{x^2 + y^2}.
\end{align*}
Therefore,
\[
dF = \frac{x}{x^2 + y^2}\, dx + \frac{y}{x^2 + y^2}\, dy = \omega.
\]
Thus, \( \omega \) is \textbf{exact}, and the de\,Rham class \( [\omega] \in H^1_{\mathrm{dR}}(U) \) is trivial.

\subsubsection*{Conclusion}

Although the domain \( U = \mathbb{R}^2 \setminus \{0\} \) is not simply connected, the 1-form \( \omega \) is exact. Hence, the differential equation
\[
\frac{x}{x^2 + y^2}\, dx + \frac{y}{x^2 + y^2}\, dy = 0
\]
is globally solvable on \( U \), with general solution given by
\[
F(x, y) = \frac{1}{2} \log(x^2 + y^2) = C.
\]
\qed
\begin{figure}[h]
 \centering
\includegraphics[width=.6\columnwidth]{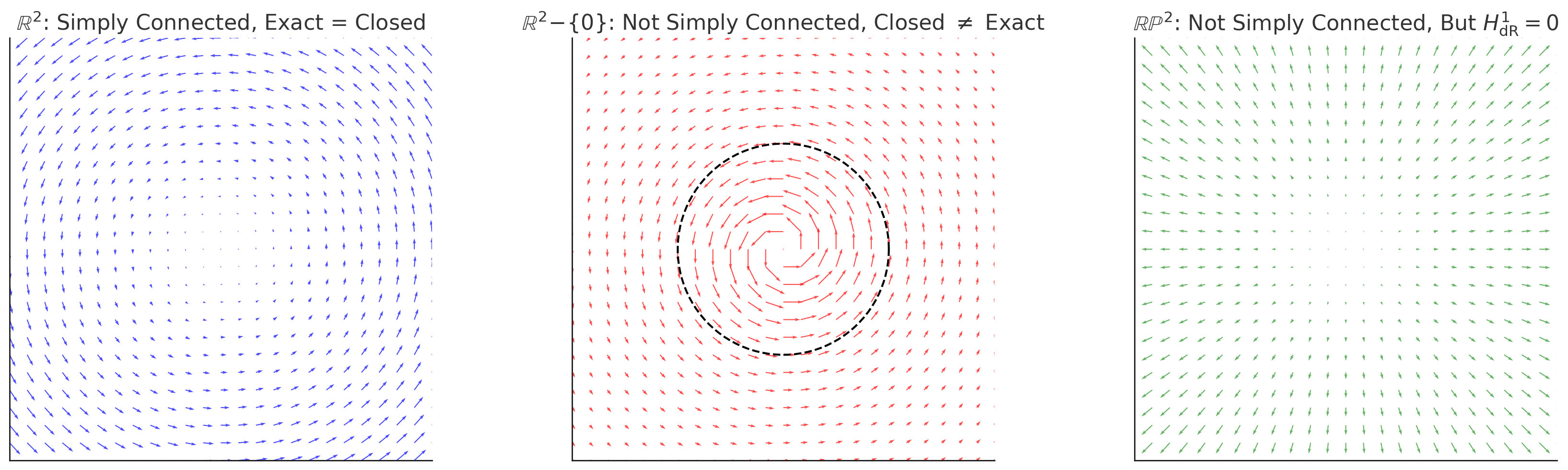}
\caption{Illustration of the relationship between topology and exactness of closed 1-forms.
        \textbf{Left:} $\mathbb{R}^2$ is simply connected; all closed forms are exact.
        \textbf{Middle:} $\mathbb{R}^2 - \{0\}$ is not simply connected; closed forms may not be exact.
        \textbf{Right:} $\mathbb{RP}^2$ is not simply connected but has trivial $H^1_{\text{dR}}$; all closed forms are still exact.}
\label{exactness-topology}
\end{figure}

\section{Conclusion}

We have seen that the integrability of a first-order differential equation is intimately linked to the exactness of its associated 1-form. By recasting the problem in the language of differential forms and de Rham cohomology, we gain a deeper understanding of both local and global aspects of solvability.

Classical methods such as integrating factors, while powerful in many practical settings, are limited by their local nature and cannot detect global topological obstructions. The cohomological framework provides a rigorous method to classify these obstructions and explain the success or failure of traditional techniques.

Moreover, the Anco--Bluman symmetry-based approach, although analytic, complements this topological understanding by systematically generating integrating factors and conserved quantities. When interpreted through the lens of cohomology, it offers a constructive pathway to resolving integrability in many cases.

This synthesis of analysis, geometry, and topology not only deepens our understanding of first-order ODEs but also paves the way for broader applications to partial differential equations, geometric flows, and the topology of dynamical systems. Future work may extend these ideas using sheaf cohomology, secondary calculus, or Hodge theory to explore integrability in higher dimensions and more complex structures.

\bibliographystyle{amsplain}

\end{document}